\documentclass[a4paper]{article}
\usepackage{graphicx} 
\usepackage{amsmath, amsthm, amsfonts, amssymb, amscd, bm, enumerate}
\usepackage{verbatim}

\usepackage{url}

\newtheorem{theorem}{Theorem}

\newtheorem{lemma}{Lemma}
\newtheorem{proposition}{Proposition}
\newtheorem{definition}{Definition}

\theoremstyle{remark}
\newtheorem{remark}{Remark}
\newtheorem{example}{Example}
\def\F{\mathcal{F}}
\def\G{\mathcal{G}}
\def\E{\mathcal{E}}
\def\H{\mathcal{H}}

\def\R{\mathbb{R}}

\def\B{\mathcal{B}}
\def\X{\mathcal{X}}
\def\SL{\mathcal{S\hspace{-0.03cm}L}}

\def\esssup[#1]{#1\text{-}\mathrm{ess\,sup}}

\begin{document}
\title{Quasi-sure analysis, aggregation and dual representations of sublinear expectations in general spaces}
\author{Samuel N. Cohen\thanks{Thanks to Terry Lyons and Freddy Delbaen for useful conversations during the preparation of this paper. In particular to F. Delbaen for the proof of Lemma \ref{lem:conditionalexpectdominance}.}\\samuel.cohen@maths.ox.ac.uk\\ University of Oxford}
\date{\today}

\maketitle
\begin{abstract}
We consider coherent sublinear expectations on a measurable space, without assuming the existence of a dominating probability measure. By considering a decomposition of the space in terms of the supports of the measures representing our sublinear expectation, we give a simple construction, in a quasi-sure sense, of the (linear) conditional expectations, and hence give a representation for the conditional sublinear expectation. We also show an aggregation property holds, and give an equivalence between consistency and a pasting property of measures.

\end{abstract}

\section{Introduction}
Decision making in the presence of uncertain outcomes is a fundamental human activity. In many cases, we need to make decisions, not only when we do not know what the outcome of our decision will be, but when we do not even know the probabilities of different outcomes. In this setting (commonly known as Knightian uncertainty, following \cite{Knight1921}) the classical mathematical approach based on the mathematical expectation is insufficient. An alternative approach in this context is to take the `worst case' under a range of different probability measures, which leads to a form of risk-averse decision making. This approach has strong axiomatic support (see Theorem \ref{thm:dualrep}) and is amenable to mathematical analysis.

When all the probability measures we consider agree on what events will occur with probability zero, this approach is, from a mathematical perspective, a relatively straightforward generalisation of the classical theory. On the other hand, when the measures do not agree in this manner (and more generally, when there is no dominating probability measure), then many difficulties arise, cutting to the heart of the mathematical theory of probability. In particular, results which are known to hold `with probability one' in the classical setting (for example, the existence and uniqueness of the conditional expectation, martingale convergence results, the martingale representation theorem, etc...) may cease to be true in this more general setting.

In some ways, this issue may seem unreasonably abstract, however it arises even in the common case of the analysis of a Brownian motion, where the volatility is known only to lie within a given bound. This problem has been studied in various frameworks by various authors, for example, Lyons \cite{Lyons1995}, Peng and coauthors \cite{Peng2010, Peng2010a, Cohen2011}, Soner, Touzi and Zhang \cite{Soner2010, Soner2010a}, Bion-Nadal and Kervarec \cite{BionNadal2010} and  Nutz \cite{Nutz2011}, amongst many others.

In this type of analysis, the detailed structure of the mathematical spaces under consideration comes to the fore, and some technical details are needed. One option is to assume that the underlying measurable space can be viewed as a separable topological space $(\Omega, \B(\Omega))$, and then to only consider those random variables  which are quasi-continuous as functions $\Omega\to\R$. This is the approach taken in Denis et al. \cite{Peng2010a}. This is in some ways unsatisfactory, as it implies that there are events (which can be easily assigned probabilities in the classicial setting) which we refuse to consider when in the setting of uncertainty, purely due to insufficient continuity. Furthermore, by results of Bion-Nadal and Kervarec \cite{BionNadal2010}, for random variables in this class there exists a dominating probability measure, that is, there exists a measure $\theta^*$ such that a (quasi-continuous) set is null for every test measure if and only if it is $\theta^*$-null. In this sense, the problem is avoided, as classical methods can be used.

A different assumption is made in Soner, Touzi and Zhang \cite{Soner2010a}, where the set of test measures is assumed to be made up of measures in a particular separable class. In particular, they consider the measures induced on Wiener space by right-constant volatility processes satisfying some further restrictions (see Example \ref{example:uncertainvol}). Under this assumption, they prove an aggregation property, with which much of the desired analysis can be performed. This approach is possibly unsatisfying as it is restricted to the problem of volatility uncertainty, and it is not apparent how this would generalise to other situations. For example, in discrete time (as one might obtain simply by taking the $\delta$-skeleton of their setting), there is no process analogous to the volatility of the Wiener process with which to parameterise the test measures, yet some regularity assumptions on the test measures are needed.

In this paper we seek to provide such regularity assumptions, in a manner consistent with \cite{Soner2010a}. We shall assume that $\Theta$, the set of test measures, permits a Hahn-like decomposition of the underlying space $\Omega$, uniformly in all the measures in $\Theta$. A key step in the proof of the main aggregation result in \cite{Soner2010a} is to verify that a stronger version of our assumption holds (our Lemma \ref{lem:hahnlem2} holds); we show that our weaker version is sufficient to guarantee their result holds (Theorem \ref{thm:aggregation}), and that with our assumption the proof is remarkably simple. On the other hand, our assumption has a natural interpretation in any space, rather than in the particular case of uncertain volatility. We shall also show that there are natural results regarding the pasting of measures and the representation of conditional sublinear expectations which follow directly from our assumption.

\section{Sublinear expectations}
The theory of sublinear expectations lies at the heart of our study. These operators can either be defined on probability spaces, when they are related to the theory of BSDEs, or can be defined using the appraoch of quasi-sure analysis, for example the $G$-expectation of Peng \cite{Peng2010} or the 2BSDEs of Soner, Touzi and Zhang \cite{Soner2010, Soner2010a}, amongst many others. In discrete time, the theory of sublinear expectations using a quasi-sure analysis is discussed in \cite{Cohen2011}. In this work, we shall use the approach of quasi-sure analysis, and shall be quite general about the types of probability spaces under consideration.

Let $(\Omega, \F)$ be a measurable space, let $m\F$ denote the $\F/\B(\R)$-measurable real valued functions. We wish to define a sublinear expectation on this space, that is, a map taking random variables to $\R$ satisfying some useful properties. We begin by defining the space of random variables for which the expectation will be well defined.

\begin{definition} \label{defn:Hdefn}
Let $\H$ be a linear space of $\F$-measurable $\R$-valued functions on $\Omega$ containing the constants. We assume that $X\in\H$ implies $|X|\in\H$ and $I_{A} X\in\H$ for any $A\in\F$.
\end{definition}

\begin{definition}\label{defn:staticexpectations}
 A map $\E:\H\to\R$ will be called a coherent sublinear expectation if, for all $X,Y\in\H$, it is
\begin{enumerate}[(i)]
 \item (Monotone:) if $X\geq Y$ (for all $\omega$) we have $\E(X)\geq \E(Y)$,
 \item (Constant invariant:) for constants $c$, $\E(c)=c$,
 \item (Cash additive:) for constants $c$, $\E(X+c)=\E(X)+c$,
 \item (Coherent:) for all constants $c>0$, $\E(cX)=c\E(X)$, and
 \item (Sublinear:) $\E(X+Y) \leq \E(X)+\E(Y)$,
 \item (Monotone continuous:) for $X_n$ a nonnegative sequence in $\H$ increasing pointwise to $X$, $\E(X_n)\uparrow \E(X)$.
\end{enumerate}
\end{definition}

Due to its convexity, a coherent sublinear expectation has a simple representation.

\begin{theorem}[See {\cite[Theorem 3.2]{Delbaen2002}, \cite[Theorem I.2.1]{Peng2010}}]\label{thm:dualrep}
A coherent sublinear expectation has a representation
\begin{equation}\label{eq:supremumrepresentation}
\E(X) = \sup_{\theta\in\Theta} E_\theta[X]
\end{equation}
where $\Theta$ is a collection of ($\sigma$-additive) probability measures on $\Omega$. For simplicity, shall say that $\Theta$ represents $\E$.
\end{theorem}

Once we have this representation, it is natural to wonder how far we can extend $\E$ to functions not in $\H$. Clearly we can define $\E$ for every bounded $\F$-measurable function. As we will not, in general, know that our measures in $\Theta$ will be absolutely continuous (in fact, the focus of this paper is on the case where they are not), we cannot simply complete $\F$ under some reference measure, however this leads us to the following definition.

\begin{definition}
 Let $\Theta$ be a collection of probability measures on $(\Omega, \F)$. Let $\F^\theta$ denote the completion of $\F$ under the measure $\theta$. We write
\[\F^\Theta=\bigcap_{\theta\in\Theta}\F^\theta.\]
The collection $\F^\Theta$ is a $\sigma$-algebra, and every $\theta\in\Theta$ has a unique extension to $\F^\Theta$.
\end{definition}

\begin{definition}
 A set $N\in\F^\Theta$ is called a \emph{($\Theta$-)polar set} if $\theta(N)=0$ for all $\theta\in\Theta$.
\end{definition}

\begin{remark}
 A natural alternative to the use of $\F^\Theta$ is to simply complete $\F$ by adding the polar sets. That is, if $\mathcal{N}$ denotes the polar sets, functions which are $\F\vee \mathcal{N}$-measurable are the main objects of study. By considering the set $\F^\Theta$, we allow a far richer class of functions, as is made clear by the following easy proposition. The $\sigma$-algebra $\F^\Theta$ is also used in \cite{Soner2010a} and \cite{Nutz2011}, where it is called the universal completion of $\F$.
\end{remark}

\begin{proposition}
For $\Theta$ a family of probability measures on $(\Omega, \F)$, where $\mathcal{N}$ denotes the $\Theta$-polar sets and $\F^\theta$ the completion of $\F$ under $\theta$,
\[\F\subseteq \F\vee\mathcal{N}\subseteq \F^\Theta \subseteq \F^\theta\]
for any $\theta\in\Theta$.
\end{proposition}

\begin{example}
 Let $\Omega=[0,1]$, $\F=\B(\Omega)$ and $\Theta=\{\delta_x\}_{x\in[0,1]}$, the set of discrete point-mass measures on $\Omega$. Then $\mathcal{N}=\{\emptyset\}$, so $\F\vee \mathcal{N}= \B(\Omega)$. However, $\F^\theta= 2^\Omega$ for all $\theta$, so $\F^\Theta=2^\Omega$. This is perfectly reasonable, as one can take the expectation of \emph{any} function under $\delta_x$ for any $x$, so there is no need to insist on any stronger concepts of measurability.
\end{example}

\begin{definition}\label{defn:HThetadefn}
 Let $\Theta$ be a collection of probability measures on $(\Omega, \F)$. We say that a function $X:\Omega\to\R$ is
\begin{itemize}
\item in $m\F^\Theta$ if it is $\F^\Theta$-measurable,
\item in $\H^\Theta_\F$ if $X\in m\F^\Theta$ and at least one of $E_\theta[X^+_\theta]$ and $E_\theta[X^-_\theta]$ is finite for all $\theta\in\Theta$, and
\item in $L^1(\E;\F)$ if $X\in m\F^\Theta$ and $\sup_\theta E_\theta[|X|]$ is finite (and similarly $L^p(\E;\F)$).
\end{itemize}
\end{definition}

We can now extend $\E$ to the larger space $\H^\Theta_\F$.

\begin{definition}
 We define the operator
\[\bar\E: \H^\Theta_\F \to \R, X\mapsto \sup_{\theta\in\Theta} E_\theta[X],\]

It is easy to verify that $\bar\E$ satisfies properties (i-iv) and (vi) of Definition \ref{defn:staticexpectations} with $\H$ replaced by $\H^\Theta_\F$, as a map $\H^\Theta_\F\to\R\cup\{\pm\infty\}$. It also satisfies property (v) provided all terms are well defined (in particular,  this is satisfied on $L^1(\E)$). Furthermore, comparing with Definition \ref{defn:staticexpectations} and Theorem \ref{thm:dualrep} we have $\H\subseteq\H^\Theta_\F$ and $\bar\E|_{\H} = \E$.
\end{definition}

Hereafter, we shall take $\Theta$ as fixed, and simply write $\H_\F$ for $\H^\Theta_\F$ and $\E$ for $\bar\E$, whenever this does not lead to confusion. However, we shall still distinguish between $\F$ and $\F^\Theta$.

\begin{definition} We say that a statement holds quasi-surely (q.s.) if it holds except on a polar set.
\end{definition}

\subsection{Conditional sublinear expectations}\label{sec:conditionalexp}
Suppose now that we have a sub-$\sigma$-algebra $\G\subseteq\F$. In exactly the same way as before (Definition \ref{defn:HThetadefn}), we can define the space $\H_\G^\Theta$, and it is easy to verify that $\H_\G^\Theta \subseteq \H_\F^\Theta$ and $\G^\Theta\subseteq\F^\Theta$. As before, we shall simply write $\H_\G$ for $\H_\G^\Theta$.

 We wish to consider the sublinear expectation conditional on $\G$. This is an operator satisfying the following properties.

\begin{definition} \label{defn:condexpdefn}
A pair of maps
\[\begin{split}
   \E&:\H_\F\to\R\\
   \E_\G&: L^1(\E;\F)\to L^1(\E;\G)
  \end{split}
\] is called a $\G$-consistent coherent sublinear expectation if for any $X,Y \in L^1(\E;\F)$
\begin{enumerate}[(i)]
\item $\E$ is a coherent sublinear expectation
\item (Recursivity) $\E \circ \E_\G = \E$ on $L^1(\E;\F)$, that is, $\E(\E_\G(X))=\E(X)$,
\item ($\G$-Regularity) $\E_\G(I_A Y) = I_A\E_\G(Y)$ q.s. for all $A\in \G^\Theta$.
\item $\E_\G$ satisfies the requirements of a coherent sublinear expectation \emph{$\G^\Theta$-conditionally}, that is
\begin{enumerate}
\item ($\G$-monotonicity) $X\geq Y$ implies $\E_\G(X)\geq \E_\G(Y)$ q.s.
\item ($\G$-triviality) $\E_\G(Y)=Y$ q.s. for all $Y\in L^1(\E;\G)$.
\item ($\G$-cash additivity) $\E_\G(X+Y)= \E_\G(X)+Y$ q.s. for all $Y\in L^1(\E;\G)$.
\item ($\G$-sublinearity) $\E_\G(X+Y) \leq \E_\G(X)+\E_\G(Y)$ q.s.
\item ($\G$-coherence) $\E_\G(\lambda Y) = \lambda^+ \E_\G(Y) + \lambda^-\E_\G(-Y)$ q.s. for all $\lambda\in m\G^\Theta$ with $(\lambda Y)\in L^1(\E;\G)$.
\end{enumerate}
\end{enumerate}
\end{definition}

The following simple lemma gives uniqueness of the conditional expectation.
\begin{lemma}
 For a given coherent sublinear expectation $\E$, a given $\G\subseteq\F$, there exists at most one conditional coherent sublinear expectation $\E_\G$, up to equality q.s.
\end{lemma}
\begin{proof}
 For a given $X$, suppose $\E_\G$ and $\bar\E_\G$ are two versions of the conditional expectation. By the $\G$-triviality and cash additivity properties, we can see that $\bar\E_\G(X-\bar\E_\G(X))=0$ q.s., and hence by regularity, for any $A\in\G^\Theta$ we have $\E(I_A(X-\bar\E_\G(X)))=0$. Similarly we see that
\[\E(I_A(\E_\G(X)-\bar\E_\G(X))) = \E(\E_\G(I_A(X-\bar\E_\G(X))))= \E(I_A(X-\bar\E_\G(X)))=0.\]
Therefore, taking $A_n=\{\omega:\E_\G(X) > \bar\E_\G(X) + n^{-1}\}\in\G^\Theta$, we have
\[0\leq \E(I_{A_n} n^{-1})\leq \E(I_{A_n}(\E_\G(X)-\bar\E_\G(X)))= 0\]
and hence $\E(I_{A_n})$ is polar. Therefore $\cup_n A_n$ is polar, that is, $\E_\G(X) \leq \bar\E_G(X)$ q.s. Reversing the roles of $\E_\G$ and $\bar\E_\G$ yields the reverse inequality.
\end{proof}

Note that, as in the classical case, we shall only require in the definition that $\E_\G$ is well defined on $L^1(\E;\F)$. However, it will often be the case (cf Remark \ref{rem:conditionaldomain}) that the conditional expectation is well defined on a space of functions with significantly less integrability.

\section{Representing the conditional expectation}
For a given $\G$-consistent sublinear expectation $\E$, we wish to have a representation of the conditional expectation $\E_\G$ similar to that in Theorem \ref{thm:dualrep}. That is, we wish to write
\begin{equation}\label{eq:supremumrepcond}
``\E_\G(X) = \sup_{\theta\in\Theta} E_\theta[X|\G]\text{.''}
\end{equation}
This statement has two key problems. First, the conditional expectation $E_\theta[\cdot|\F_t]$ is only defined $\theta$-a.s. rather than $\E$-q.s. When $\Theta$ consists of uncountably many possibly singular probability measures, this causes a significant problem. Second, if $\Theta$ is uncountable, the pointwise supremum may be an inappropriate choice, as it is unclear whether it is even in $m\G^\Theta$.

To deal with these issues, we shall first assume that our set of measures satisfies a certain decomposition property, which is a generalisation of the separability assumed in Soner et al. \cite{Soner2010a}. Under this assumption, we shall be able to give a consistent definition of the conditional expectation under $\theta$, in a quasi-sure sense. We then follow Detlefsen and Scandolo \cite{Detlefsen2005} in replacing the supremum in (\ref{eq:supremumrepcond}) with an essential supremum, which we construct quasi-surely. Hence, we show that the representation is valid. It is worth also noting the work of Bion-Nadal \cite{Bion-Nadal2006}, where a similar representation is obtained (for the larger class of convex risk measures under uncertainty, that is, without the assumption of coherence) however no consideration is given to the construction of the conditional expectation in a quasi-sure sense.

\begin{definition}
For $\G\subseteq\F$, we shall write $\Theta|_\G$ for the set of measures $\theta\in\Theta$, all restricted to $\G$.
\end{definition}
\subsection{Defining linear conditional expectations}
Our key tool for the definition of the conditional expectation, in a sufficiently strong sense, will be the assumption that the following property holds.
\begin{definition}[Hahn property]\label{def:Hahn}
We shall say that $\Theta$ has the \emph{Hahn property on $\G$} if there exists a `dominating' set of probability measures $\Phi$ defined on $(\Omega, \G)$ such that
\begin{enumerate}[(i)]
 \item $\Phi$ and $\Theta|_\G$ generate the same polar sets and $m\G^\Theta=m\G^\Phi$,
 \item for every $\phi\in\Phi$, there is  a set $S_{(\phi, \G)}\in\G^\Theta$ that supports $\phi$, that is, \[\phi(S_{(\phi, \G)}) = 1,\]
 such that the sets $\{S_{(\phi,\G)}\}_{\phi\in\Phi}$ are disjoint.
\end{enumerate}
The collection $\{S_{(\phi, \G)}\}_{\phi\in\Phi}$, with the associated measures $\Phi$, will be called a $\Theta/\G$-dominating partition of $\Omega$. (Note that $\{S_{(\phi, \G)}\}$ is a $\G^\Theta$-measurable partition of $\Omega$ minus a polar set.)
\end{definition}
Note that the `dominating' set $\Phi$ is not assumed to be countable. The reason for giving this name to the property will be outlined in Remark \ref{rem:namereason}. The following example shows that the existence of a Hahn decomposition is not trivial in general.
\begin{example}
 Consider the space $\Omega=[0,1]^2$ with its Borel $\sigma$-algebra. For simplicity, we take $\G=\B(\omega_1)$, the Borel $\sigma$-algebra generated by the first component of $\Omega$. Let $\Theta=\{\delta_{(x,y)}: (x,y)\in[0,1]^2\}$, the family of single-point measures on $\Omega$. Then $\Theta$ has the Hahn property, with $\Theta=\Phi$ and $S_{(\delta_{(x,y)}, \G)}=\{x\}\times[0,1]$. The set of measures obtained by taking all countable mixtures of elements of $\Theta$ will also have the Hahn property, a $\Theta/\G$-dominating partition being the sets $\{\{x\}\times[0,1]\}_{x\in[0,1]}$.

Conversely, if $\Theta' = \Theta \cup \{\lambda\}$, where $\lambda$ is Lebesgue measure on $[0,1]^2$, then $\Theta'$ does not have the Hahn property. This is because any dominating set $\Phi$ must generate no non-empty polar sets, and for every point $x$ there is a measure $\phi\in \Phi$ such that $\phi(x)>0$. As the supports of the measures in $\Phi$ are disjoint, $\Phi$ must be built up only of measures supported by countably many points. This implies, however, that all functions are $\G^\phi$-measurable for each $\phi\in\Phi$, so all functions are in $m\G^\Phi$. On the other hand, $m\G^\Theta$ only contains Lebesgue measurable functions, so we see that $m\G^\Phi \neq m\G^\Theta$.
\end{example}

\begin{example}
 Suppose there exists a dominating measure $\phi$ on $\G^\Theta$, that is, $\theta|_\G$ is absolutely continuous with respect to $\phi$ for all $\theta\in\Theta$ and without loss of generality $\phi(A)=0$ for all polar sets $A$. Then $\Theta$ has the Hahn property with $\Phi = \{\phi\}$.
\end{example}

The usefulness of the Hahn property is due to the following simple lemma.

\begin{lemma}\label{lem:hahnlem3}
 Let $\Theta$ have the Hahn property on $\G$ and let $A\in \G^\Theta$ with $A\subseteq S_{(\phi, \G)}$ for some $\phi\in\Phi$. Then $A$ is polar if and only if $A$ is $\phi$-null.

 Hence for every $\theta\in\Theta$, every $\phi\in\Phi$, we know $\theta|_\G$ is absolutely continuous with respect to $\phi$ on $S_{(\phi, \G)}$.
\end{lemma}
\begin{proof}
 By assumption (i) of the Hahn property, all polar sets must be $\phi$-null for every $\phi\in\Phi$. Conversely, as $A\subseteq S_{(\phi, \G)}$ and the supports $\{ S_{(\psi, \G)}\}_{\psi\in\Phi}$ are disjoint, $\psi(A)=0$ for every $\psi\neq \phi$, $\psi\in\Phi$. As $\phi(A)=0$ also, we know that $A$ is $\Phi$-polar, and hence is $\Theta$-polar.
\end{proof}

In some cases, the Hahn property may be most easily verified using the following lemma.
\begin{lemma}\label{lem:hahnlem2}
 Suppose there exists a subset $\Phi\subseteq \Theta$ with disjoint supports $\{S_{(\phi, \G)}\}_{\phi\in\Phi}$, such that for any $\theta\in\Theta$ there exists a countable set $\{\phi^\theta_n\}\subseteq\Phi$ with
\begin{itemize}
 \item $\bigcup_n S_{(\phi^\theta_n, \G)}$ supports $\theta$, and
 \item $\theta|_\G$ is absolutely continuous with respect to $\phi^\theta_n|_\G$ on $S_{(\phi^\theta_n,\G)}$.
\end{itemize}
 Then $\Theta$ has the Hahn property (and $\Phi|_\G$ is a $\Theta/\G$-dominating partition).
\end{lemma}
\begin{proof}
We only need to show that $\Phi$ and $\Theta$ generate the same polar sets in $\G$ and $m\G^\Theta=m\G^\Phi$. As $\Phi\subseteq\Theta$, any $\Theta$-polar set is clearly $\Phi$-polar and $m\G^\Phi\supseteq m\G^\Theta$.

For the converse, for any $\theta\in\Theta$, by assumption there is a countable set $\{\phi^\theta_n\}$ in $\Phi$ such that $\bigcup_n S_{(\phi^\theta_n,\G)}$ supports $\theta$. For any $\Phi$-polar $A\in\G^\Phi$, we then have
\[\theta (A) = \sum_n \theta (A\cap S_{(\phi^\theta_n,\G)}).\]
However, $\theta|_\G$ is absolutely continuous with respect to $\phi^\theta_n|_\G$ on $S_{(\phi^\theta_n,\G)}$, so if $\phi^\theta_n(A)=0$ we have $\theta (A\cap S_{(\phi^\theta_n,\G)})=0$. Hence $\theta(A)=0$, and as $\theta$ was arbitrary we know $A$ is $\Theta$-polar.

Similarly, if $X\in m\G^\Phi$, then for any $\theta\in\Theta$ we have the countable set $\{\phi_n^\theta\}$, and for each $n$, we see that $X$ differs from a $\G$-measurable function on a $\phi^\theta_n$-null set. On $S_{(\phi^\theta_n, \G)}$, we know $\theta$ is absolutely continuous with respect to $\phi^\theta_n$, so there is a $\G$-measurable function $\tilde X$ such that $\{X\neq \tilde X\}\cap S_{(\phi^\theta_n, \G)}$ is $\theta$-null. From the representation \[X=\sum_n I_{S_{(\phi^\theta_n, \G)}} X\qquad \theta-a.s.,\] we see that $X\in\G^\theta$ for all $\theta$, so $X\in m\G^\Theta$.
\end{proof}

We can now see that the setting of Soner et al. \cite{Soner2010a} has the Hahn property.

\begin{example}\label{example:uncertainvol}
Let $\Omega$ be the classical Wiener space, with canonical process $B$ starting at zero. Let $\F_t=\sigma\{B_s\}_{0\leq s\leq t}$, and $\G=\F_t$ for some $t$. Let $\langle B\rangle$ be the quadratic variation, which is a progressively measurable continuous function and can be universally defined for all local martingale measures on $B$, as in Karandikar \cite{Karandikar1995} (this is a  scalar version of the setting of \cite{Soner2010a}, see also Nutz \cite{Nutz2011}).

Consider the set of orthogonal measures $\theta$ parameterised by some subset of the $\F$-predictable absolutely continuous nonnegative functions, where under $\theta_v$, $B$ is a local martingale with quadratic variation $v$. Then we can take $S_{(\theta_v, \G)} = \{\omega:\langle B\rangle_s = v_s \text{ for all }s\leq t\}$, which is a $\G$-measurable set. Soner et al. \cite{Soner2010a} take $v$ of the form
\[\frac{dv}{dt} = \sum_{n=0}^\infty\sum_{i=0}^\infty a^i_n I_{E^n_i}I_{[\tau_n,\tau_{n+1}[},\]
where the $(a^i_n)$ come from a generating class (for example, the class of deterministic processes), the $(\tau_n)$ is an increasing sequence of stopping times taking countably many values and q.s. reaching $\infty$ for finite $n$, and $\{E^n_i\}\subset\F_{\tau_n}$ is a family of partitions of $\Omega$. Such processes $v$ are said to satisfy the separability condition.

We claim the measures associated with the generating class, restricted to $\G$, form a $\Theta/\G$-dominating partition of $\Omega$ (up to repeated sets in the partition). Under the separability condition, the measures associated with the generating class, restricted to $\G$, have either identical or disjoint supports and are included in $\Theta$. As every measure in $\Theta$ is generated by a countable collection of elements of the generating class, the first requirement of our Lemma \ref{lem:hahnlem2} is satisfied. Lemma 5.2 of \cite{Soner2010a} then proves the equivalence (in fact, the equality) of any two measures in $\Theta$ on the intersection of their supports, yielding the second condition of our Lemma \ref{lem:hahnlem2}.
\end{example}

\subsection{The essential supremum}

It is useful to be able to combine families of random variables in a quasi-surely consistent manner. A key tool for doing this is the essential supremum, which we now construct in a quasi-sure sense. To begin, we cite the following result on the existence of the essential supremum in a classical setting.

\begin{theorem}[F\"ollmer and Schied \cite{Follmer2002} (Thm A.18)]\label{thm:classicalessentialsup}
Let $\X$ be any set of $\G$-measurable random variables on a (complete) probability space $(\Omega, \G, \theta)$.
\begin{enumerate}[(i)]
 \item Then there exists a random variable $X^*$ such that $X^*\geq X$ $\theta$-a.s. for all $X\in\X$. Moreover $X^*$ is $\theta$-a.s. unique in the sense that any other random variable $Y$ with this property satisfies $Y\geq X$ $\theta$-a.s. We call $X^*$ the $\theta$-essential supremum of $\X$, and write $X^*=\esssup[\theta]\X$.
\item Suppose that $\X$ is upward directed, that is, for $X, X'\in\X$ there is $X''\in \X$ with $X''\geq X\vee X'$. Then there exists an increasing sequence $X_1\leq X_2 ...$ in $\X$ such that $X^*= \lim_n X_n$ $\theta$-a.s.
\end{enumerate}
\end{theorem}

We can extend the first half of this result to our setting, using the Hahn property.

\begin{theorem}\label{thm:quasiessentialsup}
Suppose $\Theta$ is a collection of measures with the Hahn property on $\G$. Then for any set $\X\subset m\G^\Theta$, the result of Theorem \ref{thm:classicalessentialsup}(i) holds, where all random variables are taken to be in $m\G^\Theta$, and inequalities are taken to hold q.s. For clarity, we denote the $\Theta$-q.s. essential supremum by $\esssup[\Theta]$.
\end{theorem}
\begin{proof}
 Let $\{S_{(\phi,\G)}\}$ be a $\Theta/\G$-dominating partition of $\Omega$. As $m\G^\Theta=m\G^\Phi$, we know that $X\in\X$ is $\G^\phi$-measurable for all $\phi$. Hence we can use Theorem \ref{thm:classicalessentialsup}(i) to construct the essential supremum $X^*_\phi = \esssup[\phi] \{\X\}$, and then define the `universal' essential supremum by the disjoint sum
\[X^*:= \sum_{\phi\in\Phi} I_{S_{(\phi,\G)}} X_\phi^*.\]
Clearly for any $X\in\X$ we have $X^*\geq X$ q.s. on $S_{(\phi,\G)}$ for all $\phi$, hence by Lemma \ref{lem:hahnlem2}, $X^*\geq X$ q.s. on $\Omega$. It is easy to verify that $X^*$ is unique q.s., as $X^*_\phi$ is unique q.s. for each $\phi$. To show measurability, note that $X^*\in m\G^\phi$ for all $\phi$, so $X^*\in m\G^\Phi$. As $ m\G^\Phi = m\G^\Theta$ by the Hahn property, the result is proven.
\end{proof}

We can now construct, in a q.s. unique way, the supports of the measures $\theta\in\Theta$.

\begin{definition}
 Let $\Theta$ have the Hahn property. For $\theta\in\Theta$, $\phi\in\Phi$, define
\[\lambda_{\theta|\phi} := \frac{d\theta|_\G}{d\phi}I_{S_{(\phi, \G)}}\]
where  by Lemma \ref{lem:hahnlem3} the Radon-Nikodym derivative is well defined $\phi$-a.s. on $S_{(\phi,\G)}$, and hence $\lambda_{\theta|\phi}$ is defined up to a polar set. Then define the $\G^\Theta$-measurable support of $\theta$,
\[S_{(\theta, \G)} := \{\omega: \esssup[\Theta]_{\phi\in\Phi} (\lambda_{\theta|\phi}) >0\} \in \G^\Theta.\]
\end{definition}
\begin{lemma}\label{lem:Sisminimalsupport}
 Any $\G^\Theta$-measurable $\theta$-null subset of $S_{(\theta,\G)}$ is polar.
\end{lemma}
\begin{proof}
Let $A\in\G^\Theta$ be a $\theta$-null subset of $S_{(\theta, \G)}$. If $A$ is not polar, then there exists $\phi\in\Phi$ such that $\phi(A)>0$. By the definition of $S_{(\theta,\G)}$ and the essential supremum, we know $\lambda_{\theta|\phi}>0$ $\phi$-a.s. on $S_{(\phi, \G)}\cap S_{(\theta, \G)}$, so
\[\theta(A)\geq  \int_{A} \lambda_{\theta|\phi} d\phi > 0,\]
 which implies $A$ is not $\theta$-null, giving a contradiction.
\end{proof}
\begin{remark}
Note that this lemma implies that $S_{(\theta, \G)}$ is the `smallest' $\G^\Theta$-measurable support of $\theta$, in a q.s. sense. That is, if there was another $\G^\Theta$-measurable support $R\subset S_{(\theta, \G)}$, then  we know $S_{(\theta, \G)}\setminus R \in \G^\Theta$ would be $\theta$-null, hence from the lemma it is polar.
\end{remark}
\begin{lemma}\label{lem:equivonsupport}
For any two measures $\theta, \theta'\in\Theta$, their restrictions $\theta|_\G$ and $\theta'|_\G$ are equivalent on the intersection of their supports. That is, if $A\subset S_{(\theta,\G)}\cap S_{(\theta',\G)}$, $A\in\G^\Theta$ is $\theta$-null, it is also $\theta'$-null.
\end{lemma}
\begin{proof}
 If $A$ is $\theta$-null it is polar, by Lemma \ref{lem:Sisminimalsupport}, and hence is also $\theta'$-null.
\end{proof}

\begin{remark}\label{rem:namereason}
 This lemma is the reason why we have used the name  `Hahn property'. From this lemma, we see that our assumption allows us to decompose our space into supports for our restricted measures $\Theta|_\G$ such that they are equivalent on the intersection of their supports. When we consider only two measures, this can be done using a combination of the Lebesgue decomposition theorem and the Hahn decomposition theorem. Here we assume enough that we can \emph{simultaneously} find supports for our uncountable family of measures such that the decomposition holds for all pairs, keeping the supports fixed.
\end{remark}

We can also reproduce an aggregation result similar to that of Touzi, Soner and Zhang \cite{Soner2010a}.

\begin{theorem}\label{thm:aggregation}
 Suppose $\Theta$ has the Hahn property on $\G$. Let $\{X^\theta\}_{\theta\in\Theta}$ be any family of functions such that for all $\theta, \psi\in\Theta$
\begin{itemize}
 \item $X^\theta$ is $\G^\theta$-measurable (where $\G^\theta$ is the completion of $\G$ under $\theta)$ and
 \item $X^\theta=X^\psi$ ($\theta$-a.s.) on $S_{(\theta, \G)} \cap S_{(\psi, \G)}$.
\end{itemize}
Then there exists an \emph{aggregation function} $Y$ which is $\G^\Theta$-measurable, such that $Y=X^\theta$ $\theta$-a.s. for all $\theta$.
\end{theorem}
\begin{proof}
Simply take
\[Y=\esssup[\Theta]_{\theta\in\Theta} X^\theta.\]
For any $\theta\in\Theta$, by our second assumption we see that $Y=X^\theta$ $\theta$-a.s. on $S_{(\theta,\G)}$, and as $S_{(\theta,\G)}$ supports $\theta$, $Y=X^\theta$ $\theta$-a.s.
\end{proof}

As shown in \cite{Soner2010a}, many of the results of stochastic analysis can be obtained as soon as we have a result of this kind. 

\subsection{A dual representation}
We now prove that a modified version of the representation (\ref{eq:supremumrepcond}) is valid.
\begin{lemma}\label{lem:conditionalexpectdominance}
 Let $\E$ be a $\G$-consistent sublinear expectation, with representation $\E(\cdot) = \sup_{\theta\in\Theta} E_\theta[X]$. Then for any $\theta\in\Theta$, any $X$ such that all terms are $\theta$-a.s. finite, any $t<\infty$,
\[-\E(-X|\G)\leq E_\theta[X|\G] \leq \E(X|\G)\qquad \theta-a.s.\]
\end{lemma}
\begin{proof}
For any $A\in\G^\Theta$, any $X$ we have
\[\E_\G[I_A(X-\E_t(X))] = \E_\G(I_AX)-I_A\E_\G(X)= 0\]
and so by time consistency $\E_\G[I_A(X-\E_t(X))]=0.$ Hence
\[E_\theta[I_A(X-\E_\G(X))] \leq \E[I_A(X-\E_\G(X))]=0\]
and rearrangement gives $E_\theta[I_AX] \leq E_\theta[I_A\E_\G(X))]$, which is equivalent to the upper bound $ E_\theta[X|\G] \leq \E_\G(X)$. For the lower bound, applying this result to $-X$ gives
\[E_\theta[X|\G] = -E_\theta[-X|\G] \geq -\E_\G(-X) \quad \theta-a.s.\]
\end{proof}

Using the Hahn property, we can consistently define our conditional expectations $E_\theta[\cdot|\G]$ up to equality $\E$-q.s.

\begin{definition}
Suppose $\Theta$ has the Hahn property on $\G$. For each $\theta\in\Theta$, we define 
\[E_{\theta|\Theta}[X|\G] = \begin{cases} Y & \omega \in S_{(\theta, \G)}\\ -\infty & \omega\not\in S_{(\theta, \G)}\end{cases}\]
where $Y\in m\G^\Theta$ is any version of the classical conditional expectation $E_\theta[X|\G]$. By Lemma \ref{lem:Sisminimalsupport}, this definition is unique up to a polar set (as it is unique up to a $\theta$-null subset of $S_{(\theta, \G)}$).
\end{definition}

\begin{remark}
 Note that $E_{\theta|\Theta}[X|\G]$ is a version of the usual conditional expectation, but is defined $\E$-q.s. rather than $\theta$-a.s. Furthermore, $E_{\theta|\Theta}[X|\G]$ satisfies the usual properties of the conditional expectation on $S_{(\theta, \G)}$, i.e. linearity, recursivity, monotonicity, etc., again $\E$-q.s. rather than simply $\theta$-a.s. The reason for setting the expectation to $-\infty$ off $S_{(\theta, \G)}$ is simply so that we can take the supremum in a simple manner. It also gives the following lemma.
\end{remark}

\begin{lemma}
 $E_{\theta|\Theta}[X|\G]$ is the q.s. minimal version of the $\theta$-conditional expectation. That is, if $Y\in m\G^\Theta$ is another version of the conditional expectation and $Y\leq E_{\theta|\Theta}[X|\G]$, then $\{\omega: Y<E_{\theta|\Theta}[X|\G]\}$ is polar.
\end{lemma}
\begin{proof}
By definition, $Y=E_{\theta|\Theta}[X|\G]=-\infty$ except on $S_{(\theta,\G)}$. Hence $\{\omega: Y<E_{\theta|\Theta}[X|\G]\}$ is a $\theta$-null subset of $S_{(\theta,\G)}$. By Lemma \ref{lem:Sisminimalsupport}, this set is polar.
\end{proof}

We can now prove our general representation.

\begin{theorem}\label{thm:representation}
 Let $\E$ be a $\G$-consistent sublinear expectation, with a representation $\Theta$ having the Hahn property on $\G$. Then the conditional expectation has a representation
\[\E_\G(X) = \esssup[\Theta]_{\theta\in\Theta}\{E_{\theta|\Theta}[X|\G]\}\]
up to equality q.s.
\end{theorem}
\begin{proof}
 First note that for any $A\in\G^\Theta$,
\[\begin{split}
   \E(I_A\E_\G(X))&=\E(I_AX) = \sup_{\theta\in\Theta} E_\theta[I_AX] = \sup_{\theta\in\Theta} E_\theta[I_AE_{\theta|\Theta}[X|\G]] \\
&\leq \sup_{\theta\in\Theta} E_\theta[I_A(\esssup[\Theta]_{\psi\in\Theta}\{E_{\psi|\Theta}[X|\G]\})] \\
&= \E(I_A (\esssup[\Theta]_{\psi\in\Theta}\{E_{\psi|\Theta}[X|\G]\}))
  \end{split}
\]
from which we see
\[\E_\G(X) \leq \esssup[\Theta]_{\psi\in\theta}\{E_{\psi|\Theta}[X|\G]\} \quad q.s.\]
Conversely, by Lemma \ref{lem:conditionalexpectdominance}, we know that for every $\psi\in\Theta$, $E_{\psi|\Theta}[X|\G] \leq \E_\G(X)$ $\psi$-a.s. By definition, $E_{\psi|\Theta}[X|\G]=-\infty$ except on $S_{(\psi,\G)}$, so by Lemma \ref{lem:hahnlem3} we know
\[E_{\psi|\Theta}[X|\G] \leq \E_\G(X)\quad  q.s.\]
Therefore, by Theorem \ref{thm:quasiessentialsup},
\[\esssup[\Theta]_{\psi\in\theta}\{E_{\psi|\Theta}[X|\G]\}\leq \E_\G(X)\quad q.s.\]
giving the desired equality.
\end{proof}

As mentioned earlier, Bion-Nadal \cite{Bion-Nadal2006} gives a similar result to this, however without a quasi-sure construction of the conditional expectation. Therefore, her result presents only the $\theta$-a.s. equality of the conditional sublinear expectation and the $\theta$-essential supremum. Our result is strictly stronger, as both the equality and the essential supremum are taken in a quasi sure sense.

\begin{remark}\label{rem:conditionaldomain}
 We note that this result immediately allows us to consistently extend $\E_\G$ to the larger space $\H_\F$, using a generalised conditional expectation, as in \cite[p2]{Jacod2003}. That is, we no longer require substantial integrability conditions on $X$ to define $\E_\G(X)$. This will, however, lead to somewhat different statements of the properties of the conditional expectation (as finiteness is no longer guaranteed).
\end{remark}

\subsection{$\G$-consistency and pasting of measures}

Using this result, we can give a type of `pasting stability' of the measures related to $\G$-consistency. This is closely related to the $m$-stability of Delbaen \cite{Delbaen2006a}.

\begin{definition}
 For $\Theta$ with the Hahn property, we say $\Theta$ is \emph{stable under $\G$-pasting} if for any $\theta, \theta' \in \Theta$, any $A\subseteq S_{(\theta, \G)} \cap S_{(\theta',\G)}$, $A\in \G^\Theta$ we have $\psi\in\Theta$, where $\psi$ is the measure on $\Omega$ with
\[\psi(B) := E_\theta[I_AE_{\theta'}[I_B|\G]+I_{A^c}I_B].\]
 For a set $\Theta$, we can define $\Theta^{\G}$, the \emph{finite $\G$-stabilisation of $\Theta$}, as the set of all measures obtained from $\Theta$ through finitely many combinations of this form. Clearly if $\Theta$ has the Hahn property on $\G$ then so will $\Theta^\G$.
\end{definition}
Note that this pasting only needs to hold for $A$ in the intersection of the minimal supports of the two measures. By Lemma \ref{lem:equivonsupport},  $\theta|_\G$ and $\theta'|_\G$ are equivalent on the intersection of their minimal supports, and hence the (classical) conditional expectation can be used without difficulty.

In some applications the analogous stabilisation where countably many combinations are permitted may be of interest (particularly if we wish for the supremum to be attained), however the finite case will be sufficient for our result.

\begin{theorem}
 Let $\E$ be a sublinear expectation with representation $\Theta$. Suppose $\Theta$ has the Hahn property on $\G$. Then
\begin{enumerate}[(i)]
 \item If $\E$ is $\G$-consistent, then $\E$ has an equivalent representation 
\[\E(X) = \sup_{\theta\in\Theta^\G}E_\theta[X].\]
\item If $\Theta=\Theta^\G$ then $\E$ is $\G$-consistent.
\end{enumerate}
\end{theorem}
\begin{proof}
 \emph{(i)} Suppose $\E$ is $\G$-consistent. Clearly $\Theta\subseteq \Theta^\G$, and so
\[\E(X) \leq \sup_{\theta\in \Theta^\G} E_\theta[X].\]
Conversely, for any $\psi\in\Theta^\G$ we know $\psi$ is of the form
\[\psi(B) = E_\theta\left[\sum_n I_{A_n}E_{\theta_n}[I_B|\G]\right]\]
for some finite partition $\{A_n\}$ of $\Omega$, and some measures $\theta$ and $\theta_n$ in $\Theta$. Then
\[E_\psi[X]=E_\theta\left[\sum_n I_{A_n}E_{\theta_n}[I_B|\G]\right] \leq \sup_\theta E_{\theta}[\esssup[\Theta]_{\theta} E_{\theta|\Theta}[X|\G]] = \E(X)\]
and so
\[\E(X) \geq \sup_{\theta\in \Theta^\G} E_\theta[X].\]

\emph{(ii)} As $\Theta=\Theta^\G$ has the Hahn property, for each fixed $X\in L^1(\E,\F)$ we can define the putative sublinear conditional expectation
\[\tilde\E_\G(X) := \esssup[\Theta]_{\psi\in\Theta} E_{\psi|\Theta}[X|\G].\]
 All the properties of a $\G$-consistent sublinear expectation are trivial to verify except recursivity.

To show recursivity, first select some $\theta\in\Theta$. The quasi-sure essential supremum given by Theorem \ref{thm:quasiessentialsup} must also be a version of the $\theta$-a.s. essential supremum given by Theorem \ref{thm:classicalessentialsup}. As $E_{\psi|\Theta}[X|\G]=-\infty$ except on $S_{(\psi, \G)}$, and $\Theta=\Theta^\G$, we see that
\[\tilde\E_\G(X) = \esssup[\theta]_{\psi\in\Theta, A\subseteq S_{(\theta, \G)} \cap S_{(\psi,\G)}} \{I_{A}E_{\psi|\Theta}[X|\G] + I_{A^c} E_\theta[X|\G]\} \quad \theta-a.s.\]
Furthermore, the family $\{I_{A}E_{\psi|\Theta}[X|\G] + I_{A^c} E_\theta[X|\G]\}$ is upward directed (up to equality $\theta$-a.s.). By Theorem \ref{thm:classicalessentialsup}(ii), we can then find appropriate sequences $\psi_n^\theta$, $A^\theta_n$ such that
\[\{I_{A^\theta_n}E_{\psi_n^\theta|\Theta}[X|\G] + I_{(A^\theta_n)^c}E_\theta[X|\G]\} \uparrow \tilde\E_\G(X) \quad \theta-a.s.\]

We now relax our selection of $\theta$, and consider the equation
\[\begin{split}
   \E(\tilde\E_\G(X)) &= \sup_\theta E_\theta[ \tilde\E_\G(X)]\\
&= \sup_\theta E_\theta[\lim_n\{I_{A^\theta_n}E_{\psi_n^\theta|\Theta}[X|\G] + I_{(A^\theta_n)^c}E_\theta[X|\G]\}]\\
&= \sup_\theta \sup_n E_\theta[I_{A^\theta_n}E_{\psi_n^\theta|\Theta}[X|\G] + I_{(A^\theta_n)^c}E_\theta[X|\G]]\\
&= \sup_\theta \sup_n E_{\theta_n}[X].\\
  \end{split}
\]
where
\[\theta_n(B) := E_\theta[I_{A^\theta_n}E_{\psi_n^\theta|\Theta}[I_B|\G] + I_{(A^\theta_n)^c}E_\theta[I_B|\G]].\]
As we know $\Theta=\Theta^\G$, all the induced measures $\theta_n$ are in $\Theta$. Therefore we have
\[\E(\tilde\E_\G(X)) = \sup_\theta E_\theta[X] = \E(X)\]
and so $\tilde\E_\G(X)$ satisfies the recursivity assumption.
\end{proof}

\section{Integrability and convergence}
We now seek to look at some consequences of this representation. In particular, we shall use the representation of Theorem \ref{thm:representation} to show that, if we have a filtration $\{\F_t\}$ and we can consistently define our sublinear expectation for $\F_t^\Theta$-measurable random variables for any $t<\infty$, then we can consistently define our sublinear expectation for all $\F_{\infty-}^\Theta$-measurable random variables.

For simplicity, we shall assume that time is discrete. However, as we shall make no significant assumptions on the structure of the filtration (beyond the Hahn property), this allows consideration of the `skeleton' of a continuous filtration (possibly at stopping times) with no difficulties.

\begin{definition}\label{defn:SLexpectation}
Let $\{\F_t\}$ be a discrete-time filtration on a measurable space $(\Omega, \F)$. A family of maps
\[\begin{split}
   \E&:\H_\F\to\R\\
   \E_t&: L^1(\E;\F)\to L^1(\E;\F_t)
  \end{split}
\] is called a $\{\F_t\}$-consistent coherent sublinear expectation (on $\F$) if
\begin{enumerate}[(i)]
\item $\E$ is a coherent sublinear expectation and $\E(X)=\E_0(X)$ q.s.
\item (Recursivity) For $s\leq t$ we have $\E_s \circ \E_t = \E_s$ on $L^1(\E;\F)$, that is, $\E_s(\E_t(X))=\E_s(X)$ for all $X\in \H_\F$,
\item ($\F_t$-Regularity) $\E_t(I_A X) = I_A\E_t(X)$ q.s. for all $A\in \F_t^\Theta$, $X\in \H_F$.
\item For all $t$, $\E_t$ satisfies the requirements of a coherent sublinear expectation \emph{$\F_t^\Theta$-conditionally} (as in Definition \ref{defn:condexpdefn}).
\end{enumerate}
A $\{\F_t\}$-consistent sublinear expectation (on $\F$) will, for simplicity, be called an $\SL$-expectation (on $\F$).
\end{definition}

\begin{definition}
 We say that $\Theta$ has the Hahn property on $\{\F_t\}$ if it has the Hahn property on $\F_t$ for all $t$.
\end{definition}

To obtain convergence results, the following concepts are useful, and can be found in \cite{Cohen2011}. For simplicity, we write $L^p$ for $L^p(\E, \F)$.

\begin{definition}
 Consider $K\subset L^1$. $K$ is said to be uniformly integrable (u.i.) if
$\E(I_{\{|X|\geq c\}} |X|)$ converges to $0$ uniformly in $X\in K$ as $c\to
\infty$.
\end{definition}

\begin{definition}
 Let $L^p_b$ be the completion of the set of bounded functions $X\in \H$,
under the norm $\|\cdot\|_p=\E(|\cdot|^p)^{1/p}$. Note that $L^p_b\subset L^p$.
\end{definition}

\begin{lemma}
For each $p\geq1$,
\[L^p_b = \{ X\in L^p : \lim_{n\to \infty} \E(|X|^p I_{\{|X|>n\}})=0\}.\]
\end{lemma}

\begin{definition}
 We say $X_n$ is a uniformly integrable $\E$-submartingale if $X_n$ is $\F_n^\Theta$-measurable for all $n$, $\{X_n\}$ is uniformly integrable and $X_n\leq \E_n(X_{n+1})$ q.s. for all $n$. Similarly we define $\E$-supermartingales and $\E$-martingales.
\end{definition}

In \cite{Cohen2011}, we have obtained the following convergence result in this space.

\begin{theorem} \label{thm:supermartconvergence}
 Let $\{X_n\}_{n\in\mathbb{N}}$ be a uniformly integrable $\E$-submartingale. Then $X_n$ converges quasi surely and in $L^1(\E)$ to some random variable $X_\infty$. Furthermore, the process $\{X_n\}_{n\in \mathbb{N}\cup
\{\infty\}}$ is also a uniformly integrable $\E$-submartingale. In particular, this implies that $X_\infty \in L^1_b$. The same result holds for $\E$-supermartingales and $\E$-martingales.
\end{theorem}

\begin{lemma}\label{lem:universal test measures}
 For each $T\in\mathbb{N}$, let $\E^T_t(\cdot)$ be an $\SL$-expectation on $\F_T$, with the consistency property that
\[\E^T_t(X) = \E^{T'}_t(X)\quad \text{ for all }X\in\H_{\F_{T'}},\quad \text{all }t\leq T'\leq T\]
Then there exists a set of test probability measures $\Theta$ such that $\E^T(X) = \sup_{\theta\in\Theta}E_\theta[X]$ for all $T$.
\end{lemma}
\begin{proof}
For each $n$, let $\Theta_n$ be a set of test measures for $\E^n$. By the assumed consistency property, for $m\leq n$, for any $X\in\H_{\F_m}$,
\[\sup_{\theta\in\Theta_n}E_\theta[X] = \E^n(X) = \E^m(X) = \sup_{\theta\in\Theta_m}E_\theta[X]\]
that is, $\Theta_n$ is also a valid set of test measures for $\E^m$. It follows that, without loss of generality, we can take $\Theta'_m = \bigcup_{n\geq m} \Theta_n$, as a set of test measures for $\E^m$. Furthermore, as $\E^m(X)=\sup_{\theta\in\Theta_n}E_\theta[X]$ does not vary with $n\geq m$ and $\Theta'_n$ is nondecreasing, we have,
\[\E^m(X) = \inf_{n\geq m} \sup_{\theta\in\Theta'_n} E_\theta[X] = \sup_{\theta\in\bigcap_{k\geq m}\Theta'_k} E_\theta[X],\]
that is, without loss of generality, $\Theta''_m = \lim\sup_{n\geq m}\Theta_n = \bigcap_{k\geq m} \bigcup_{n\geq k} \Theta_n$ is also a set of test measures for $\E^m$. As the $\lim\sup_{n\geq m}$ is independent of $m$, it follows that there exists a single set of test measures $\Theta$ such that $\E^m(X) = \sup_{\theta\in\Theta} E_\theta[X]$ for all $m$.
\end{proof}

\begin{remark}
The set $\Theta$ constructed in the previous Lemma is generally not unique, as the finite-time expectations $\E^m$ can give us no information on the `correct' capacity $\E(I_A)$ associated with an event in the tail $\sigma$-algebra.
\end{remark}

Using our convergence result and our earlier representation, we can now show how to extend a finite-horizon $\SL$-expectation to an infinite horizon.
\begin{theorem}\label{thm:SLextensiontheorem}
 Let $\E^T$ and $\Theta$ be as in Lemma \ref{lem:universal test measures}, and suppose $\Theta$ has the Hahn property on $\{\F_t\}$. Then the operator
\[\E(X):=\sup_{\theta\in\Theta}E_\theta[X]\]
is an $\SL$-expectation on $L^1_b(\E;\F_{\infty-})$. In particular, it is time consistent.
\end{theorem}
\begin{proof}
 Let $\E_s:= \esssup[\Theta]_{\theta\in\Theta}\{E_{\theta|\Theta}[X|\F_s]\}$. Then it is easy to show that, for any $r\leq s\leq t<\infty$, $\E_r(\E_s(X))=\E_r(X)$ for all $X\in L^1(\E;\F_t)$. We need to show this holds for $t=\infty$.

 For $X\in L^1(\E;\F_{\infty-})$, we know that $\E_t(X)$ is an $\E$-martingale, hence $\E_t(X)\to X$ q.s. and in $L^1(\E)$, by Theorem \ref{thm:supermartconvergence}. Define the operator
\[\E_s^{(t)}(X): = \E_s(\E_t(X))\]
and then by continuity of $\E_s$ in $L^1(\E)$ we have $\E_s(X) = \lim_t\E_s^{(t)}(X)$ where the limit holds q.s. and in $L^1(\E)$. Hence we have
\[\begin{split}
   \E_r(\E_s(X)) &= \E_r(\lim_t\E_s^{(t)} (X)) = \lim_t \E_r(\E_s (\E_t(X))) \\
&= \lim_t \E_r(\E_t(X)) = \lim_t \E_r^{(t)}(X) = \E_r(X),
  \end{split}
\]
where the third inequality is because recursivity holds for finite-time-measurable variables.
\end{proof}
\section{Conclusion}
We have considered sublinear expectations on general probability spaces, where the set of measures in the dual representation of the expectation are not necessarily absolutely continuous with respect to any dominating measure. In this context, we have shown that the assumption of a Hahn property provides a simple means to aggregate processes defined with respect to each measure, thereby giving a straightforward approach to quasi-sure analysis in this context.

Our methods generalise the approach of \cite{Soner2010a}, as the Hahn property has a natural interpretation in a general setting. Consequently, this paper provides a quasi-sure construction of the conditional expectation under each test measure, and shows that a dual representation then holds for the conditional sublinear expectation. We have given a version of the aggregation result of \cite{Soner2010a}.

For any specific problem, determining whether the Hahn property holds may be a difficult task, (as is made clear by the analysis in \cite{Soner2010a}). However, our approach shows that for any given problem, once the Hahn property has been shown, many of the results of stochastic analysis transfer simply into a quasi-sure setting.

\bibliographystyle{plain}
\bibliography{../RiskPapers/General}
\end{document}